\documentclass[final]{siamltex}

\usepackage{algorithmic,algorithm}
\usepackage[leqno]{amsmath}
\usepackage{amsfonts, amssymb, amscd, mathrsfs}
\usepackage{accents}
\usepackage{bm}
\usepackage[pdftex]{graphicx}

\newcounter{remark-count}
\newenvironment{remark}[1][Remark]{ \refstepcounter{remark-count} \begin{trivlist} \item[\hskip \labelsep{ \hskip 0.5cm  \itshape #1 \thesection.\arabic{remark-count}.}]   }{\end{trivlist}}

\numberwithin{remark-count}{section}
\numberwithin{algorithm}{section}

\title{Combined Preconditioning with Applications in
  Reservoir Simulation}

\author{ 
  Xiaozhe Hu \thanks{Department of Engineering Mechanics, Kunming University of Science and Technology, China; Department of Mathematics, The Pennsylvania
    State University, University Park, PA 16802. Email:
    \texttt{hu\_x@math.psu.edu}} 
 \and Shuhong Wu \thanks{RIPED, PetroChina Company Limited, China.
 Email: \texttt{wush@petrochina.com.cn}}
\and Xiao-Hui Wu \thanks{ExxonMobil
    Upstream Research Company, USA. Email:
    \texttt{xiao-hui.wu@exxonmobil.com}} 
\and Jinchao Xu
  \thanks{Department of Mathematical Sciences, The Pennsylvania State
    University, University Park, PA 16802. Email:
    \texttt{xu@math.psu.edu}. The research of this author was supported in part by NSF Grant DMS-0915153, DOE Grant DE-SC0006903, LLNL Subcontract B596235, and NSFC Grant 91130011/A0117.} 
\and Chen-Song Zhang \thanks{NCMIS \& LSEC, Academy of Mathematics and System Sciences, Beijing, China 100190. Email: \texttt{zhangcs@lsec.cc.ac.cn}. This author is supported by the Dean Startup Fund of NCMIS, Academy of Mathematics and System Sciences.} 
\and Shiquan Zhang
  \thanks{Department of Mathematics, Sichuan University, China. Email:
    \texttt{shiquanz3@gmail.com}} 
\and Ludmil Zikatanov
  \thanks{Department of Mathematics, The Pennsylvania State
    University, University Park, PA 16802. Email:
    \texttt{ltz@math.psu.edu}} }

\begin{document}

\maketitle

\begin{abstract}
We develop a simple algorithmic framework to solve large-scale symmetric positive definite linear systems.  At its core, the framework relies on two components: (1) a norm-convergent iterative method (i.e. smoother) and (2) a preconditioner.  The resulting preconditioner, which we refer to as a combined preconditioner, is much more robust and efficient than the iterative method and preconditioner when used in Krylov subspace methods.  We prove that the combined preconditioner is positive definite and show estimates on the condition number of the preconditioned system.  We combine an algebraic multigrid method and an incomplete factorization preconditioner to test the proposed framework on problems in petroleum reservoir simulation.  Our numerical experiments demonstrate noticeable speed-up when we compare our combined method with the standalone algebraic multigrid method or the incomplete factorization preconditioner.
\end{abstract}



\pagestyle{myheadings}
\thispagestyle{plain}

\section{Introduction} \label{sec:intro} 
We consider a sparse linear system of equations that arises from the discretizations of the elliptic partial differential equations (PDEs) used to simulate the physical processes in petroleum reservoirs.  The Petroleum Reservoir Simulation (PRS) is a tool for predicting hydrocarbon reservoir performance under various operating regimes. PRS helps engineers to obtain information pertaining to the processes that take place within oil reservoirs -- information that can be used to maximize recovery and minimize environmental damage.  The crucial aspect of PRS is its ability to solve large-scale discretized PDEs, which are strongly coupled, indefinite and often non-symmetric.  Solving these linear systems consumes most of the computational time in all modern reservoir simulators -- more than 75\% of the computational time in general.  Furthermore, the demand for more accurate simulations has led to larger discrete reservoir models. And, this increase in model size results in larger linear systems, that are more difficult, or even impossible to solve in an acceptable amount of time using standard direct or standard iterative solvers. 

Over the last 30 years,  incomplete LU factorization (ILU) has become one of the most commonly used methods for solving large sparse linear equations arising in PRS. First developed in the 1960s~\cite{Buleev.N1960, Varga.R1962}, ILU methods provide an approximation of the exact LU factorization (computed via Gaussian elimination) by specifying the sparsity of the $L$ and $U$ factors. These methods need not be convergent when used in linear iterative procedures, but they do provide preconditioners that are used in Krylov subspace methods.  Thought to have introduced the word \emph{preconditioning}, Evans \cite{Evans.D1968} may also have been the first to use ILU as a preconditioner.  As noted, due to their simplicity, ILU methods are of particular interest to researchers in the field of reservoir simulation. However, when ILU preconditioners are applied to problems in PRS, their performance usually deteriorates as the number of grid-blocks increases.

To solve discretized scalar elliptic PDEs (Poisson-like PDEs), multigrid (MG) methods are efficient and provide a solution in optimal time and memory complexity in cases that allow standard coarse spaces to be used (see \cite{Hackbusch.W1985, Xu.J1992, Bramble.J1993,Trottenberg.U;Oosterlee.C;Schuller.A2001} and references therein for details).  However, the multigrid methods with standard coarse spaces make extensive use of the analytic information from the discretized equation and geometric information explicitly related to the discretization. This makes such methods difficult to use, such that in practice more sophisticated methods, such as algebraic multigrid (AMG) methods are preferred.  There are many different types of AMG methods -- the classical AMG (\cite{Stuben.K1983,
  Brandt.A;McCormick.S;Ruge.J1985}), smoothed aggregation AMG
(\cite{Vanek.P;Mandel.J;Brezina.M1996,
  Brezina.M;Falgout.R;MacLachlan.S;Manteuffel.T;McCormick.S;Ruge.J2005}),
AMGe (\cite{Jones.J;Vassilevski.P2001,Lashuk.I;Vassilevski.P2008}),
etc. However, despite their differences, they generally do not require geometric information from the grids.  Interested readers can refer to \cite{Ruge.J;Stuben.K1987, Wagner.C1999, Brandt.A2000, Cleary.A;Falgout.R;Henson.V;Jones.J;Manteuffel.T;McCormick.S;Miranda.G;Ruge.J2000, Falgout.R2006, 2008VassilevskiP-aa} and references therein for details on AMG methods.  Due to their efficiency, scalability, and applicability, AMG methods have become increasingly popular in practice (see e.g., \cite{Trottenberg.U;Clees.T2001}) including in modern petroleum reservoir simulations~\cite{Stuben.P;Chmakov.S2003,Clees.T;Ganzer.L.2010a,Dubois.O;Mishev.I;Zikatanov.L2009,Jiang.Y;Tchelepi.H2009,Hu.X;Liu.W;Qi.G;Xu.J;Yan.Y;Zhang.C2011a}.  However, the performance and efficiency of MG methods may degenerate as the physical and geometric properties of the problems become more complex. In such circumstances,  there is an increase in AMG-setup time (the time needed to construct coarse spaces and operators) and superfluous fill-in in the coarse grid operators, which makes applying relaxation more expensive on coarser levels.


In this paper, we provide a simple and transparent framework for combining preconditioners like ILU or Jacobi or additive Schwarz preconditioners with AMG or another norm-convergent iterative method for large-scale symmetric positive definite linear system of equaitons. In the combined preconditioners we propose, the component provided by the norm-convergent iterative method need not be very effective when used alone.  However, as we shall see, when it is paired with an existing preconditioner (such as an ILU preconditioner), the result is an efficient method. We show that the combined preconditioner is positive definite (SPD) under some mild assumptions. This guarantees the convergence of the conjugate gradient (CG) method for the preconditioned system.  We apply the combined preconditioner to petroleum reservoir simulations that involve highly heterogeneous media and unstructured grids in three spatial dimensions (with faults and pinch-outs). The numerical results show an improved performance that justifies using the combined preconditioners in PRS.

The rest of the paper is organized as follows. The algorithmic framework is introduced and analyzed in Section~2. As an example in Section~3, we formally present the combined preconditioner by describing its essential components, e.g., an ILU preconditioner and an AMG method. In Section~4, we recall the mixed-hybrid finite element method used to provide a discretization of an important component of reservoir simulation -- an elliptic PDE with heterogenous coefficients. Numerical results showing the improvement in performance are presented in Section~5, and some concluding remarks are given in Section~6.

\section{Definition of the combined preconditioner}
Consider the linear system
\begin{equation}\label{eqn:linear_system}
Au=f,
\end{equation}
with an SPD coefficient matrix $A$. Let $(\cdot, \cdot)$ be an inner product on a finite-dimensional Hilbert space $V$; its induced norm is $\| \cdot \|$. The adjoint of $A$ with respect to $(\cdot, \cdot)$, denoted by $A^T$, is defined by $(Au,v) = (u,A^Tv)$ for all $u, v\in V$.  $A$ is SPD if $A^T=A$ and $(Av,v) > 0$ for all $v \in V\backslash\{0\}$.  As $A$ is SPD with respect to $(\cdot,\cdot)$, the bilinear form $(A\cdot,\cdot)$ defines an inner product on $V$, denoted by $(\cdot,\cdot)_{A}$, and the induced norm of $A$ is denoted by $\| \cdot \|_{A}$. 

In what follows, we also need an operator $S$, referred to as a smoother. In some of the results, we distinguish two cases: (1) \emph{norm-convergent smoother}, i.e., $(I-SA)$ is a contraction in
$\|\cdot\|_A$-norm; and (2) more generally, a \emph{non-expansive smoother}, i.e., $(I-SA)$ is a non-expansive operator in $\|\cdot\|_A$ norm.  Note that \emph{norm-convergent} implies \emph{non-expansive}.

Next, in~Algorithm~\ref{alg:barB}, we combine a non-expansive iterative method $S$ with an existing SPD preconditioner~$B$, and as a result we get a preconditioner $B_{\text{co}}$.

\begin{algorithm}
\caption{}\label{alg:barB}
\flushleft{Given $S$, $B$, and $u^{k,0}=u^k$, the new iterate $u^{k+1}$ is
obtained by the following steps: }
\begin{equation*}
\begin{array}{rrcllrrcl}
   (1) & u^{k,1} &=& u^{k,0} + S (f - Au^{k,0}); 
& (2) & u^{k,2} &=& u^{k,1} + B (f - Au^{k,1});\\
   (3)  & u^{k+1} &=& u^{k,2} + S^T(f - Au^{k,2}). &&&&
\end{array}
\end{equation*}
\end{algorithm}
It is easy to see that Algorithm~\ref{alg:barB} defines a linear operator $B_{\text{co}}$ and that
\begin{equation} \label{eqn:barB}
I - B_{\text{co}}A = (I-S^TA)(I-BA)(I-SA).
\end{equation}
From~\eqref{eqn:barB}, it follows that
\begin{align}
B_{\text{co}}  & = \widetilde{S} + (B - S^TAB - BAS + S^TABAS) \nonumber
\\
	& = \widetilde{S} + (I - S^TA) B (I - AS). \label{eqn:tildeB1}
\end{align}
Here $\widetilde{S}$ is the symmetrization of $S$, defined as
\begin{equation*}
I - \widetilde{S}A = (I - S^TA)(I-SA),\quad\quad\widetilde{S}=
S+S^T-S^TAS. 
\end{equation*}
We now explain the properties pertaining to the notions of \emph{norm-convergent} and
\emph{non-expansive} smoothers, which we refer to in the subsequent sections. 
\begin{itemize}
\item \textbf{Norm-convergent smoother}.  In this case, $(I-SA)$ is a contraction in $\|\cdot\|_A$-norm; i.e., there exists $\rho \in[0,1)$, such that for all $v\in V$ we have
\begin{equation}\label{eq:rho}
\| (I - SA) v \|^{2}_{A} \leq \rho \| v \|^{2}_{A}.
\end{equation}
Note that under this assumption, $S$ is invertible. Indeed, we have,
$$
\|SAv\|_A=\|(I-(I-SA))v\|_A\ge \|v\|_A-\|(I-SA)v\|_A\ge (1-\sqrt{\rho})\|v\|_A,  
$$ 
which means $SA$ is coercive and implies that $S$ is invertible. Recall also that~\eqref{eq:rho} holds if and only if the symmetrization $\widetilde{S}$ is SPD; i.e.,
$$
\|I-SA\|_A<\rho, \quad\mbox{if and only if}\quad (\widetilde{S}v,v)>0,
\quad\mbox{for all}\quad v\in V \setminus\{0\}. 
$$
\item \textbf{Non-expansive smoother}. This is a more general case in which we assume that
\begin{equation}\label{eq:norho}
\| (I - SA) v \|^{2}_{A} \leq \| v \|^{2}_{A} \quad \forall v \in V.
\end{equation}
One important difference between these two cases is this: with the non-expansive case, $\widetilde{S}$~can be positive \emph{semi}-definite, whereas in the norm-convergent case, $\widetilde{S}$ is positive definite. 
\end{itemize}
The following result shows that the operator $B_{\text{co}}$ is SPD and, therefore, can be employed as a preconditioner for the preconditioned conjugate gradient (PCG) method. A preliminary version of the following theorem can be found in \cite{Xu.J2010}.

\begin{theorem}\label{thm:B-is-SPD}
Assume that $S: V \rightarrow V$ is a non-expansive smoother.
Moreover, the operator $B: V \rightarrow V$ is SPD. Then, $B_{\text{co}}$ defined  in~\eqref{eqn:barB}~is SPD.
\end{theorem}
\begin{proof}
For any $x \in V$, $x\neq 0$, we define $\tilde x := (I-SA)x$. Thus, we obtain
\begin{eqnarray*}
((I-B_{\text{co}}A)x,x)_{A} & = &((I-BA)(I-SA)x, (I-SA)x)_A   \\
			   & = & ((I-BA) \tilde x, \tilde x)_A =
                           (\tilde x, \tilde x)_A - (A\tilde x,A\tilde x)_B
\end{eqnarray*}
where $(\cdot, \cdot)_B$ is an inner product on $V$ defined as $(\cdot, \cdot)_B:=(B\cdot, \cdot)$ since $B$ is SPD. 

Further, if $\tilde x = 0$, then $((I-B_{\text{co}}A)x, x) _A = 0< \|x\|_A^2$.  On the other hand, for $\tilde x \neq 0$, we obtain from $(A\tilde x,A\tilde x)_B > 0$ that
\begin{equation*}
\|x\|_A^2-(B_{\text{co}}Ax,x)_A = 
((I-B_{\text{co}}A)x,x)_A <  \|\tilde x\|_A^2 \le \|x\|_A^2,
\quad\Longrightarrow\quad
(B_{\text{co}}Ax,Ax) > 0.
\end{equation*}
As $x\in V$ is an arbitrary nonzero element, and $A$ is non-singular, we conclude that $B_{\text{co}}$ is SPD.  Finally, in view of \eqref{eqn:tildeB1} and because both $A$ and $B$ are symmetric, it is easy to see that $B_{\text{co}}$ is symmetric. This completes the proof.
\end{proof}

\begin{remark}
Note that the assumption in the theorem is that the error transfer for the convergent iteration $I-SA$ is non-expansive in the $A$ norm and does not require $I-SA$ to be a contraction. 
\end{remark}
\begin{remark} \label{rem:checkB} Theorem~\ref{thm:B-is-SPD}~shows that the operator $B_{\text{co}}$ is SPD.  Hence, we can use the PCG method to solve~\eqref{eqn:linear_system}~with $B_{\text{co}}$ as a preconditioner, and in exact arithmetic, PCG method will always be convergent.  Algorithm~\ref{alg:barB}~is relatively simple, but the order in which $S$ and $B$ are applied is crucial to the positive definiteness  property of $B_{\text{co}}$.  For example, consider the operator $\overline{B}$ defined by
\begin{equation} \label{eqn:tildeB}
I - \overline{B}A = (I-BA)(I-\widetilde{S}A)(I-BA).
\end{equation}
In this case, it is not true in general that $\overline{B}$ is positive definite.  For example, if we use the ILU method to define a preconditioner $B$, finding the right scaling to ensure the positive definiteness of $\overline{B}$ could be a difficult task, and with indefinite $\overline{B}$ the PCG will not converge.
\end{remark}

\begin{remark} \label{rem:alg_iter} Note that directly using Algorithm~\ref{alg:barB}~as an iterative method may not result in a convergent method. This is because though we assume that $B$ is only SPD, we do not assume that $I-BA$ is a contraction.  The contraction property of $I-BA$ (and also of $I-B_{\text{co}}A$) is that any eigenvalue of $BA$ satisfies $0<\lambda(BA)\leq \omega < 2$.  However, we do not assume that such a contraction property holds for $B$.
\end{remark}


\subsection{Other derivations of the combined preconditioner}
Here, we present two distinct points of view that can lead to algorithms such as Algorithm~\ref{alg:barB}. One is from the point of view of the fictitious or auxiliary space techniques, developed in \cite{Nepomnyaschikh.S1992}~and~\cite{Xu.J1996}, and another is from the point of view of the block factorization techniques's, developed in~\cite[Chapter~5]{2008VassilevskiP-aa}.  Below, we derive Algorithm~\ref{alg:barB} using these techniques. Of course, all three derivations lead to one and the same method and  they are equivalent to each other.  However, they present different points of view and help us understand different aspects of the method.

\subsubsection{Derivation via the auxiliary space method}  
An additive version of the Algorithm~\ref{alg:barB} was derived via fictitious or auxiliary space techniques developed in \cite{Nepomnyaschikh.S1992}~and~\cite{Xu.J1996}.  Equivalently, Algorithm~\ref{alg:barB} can be viewed as a successive or multiplicative version of the auxiliary space additive preconditioner described below. Here, we use only one auxiliary space, which will turn out to be the same as $V$, but with a different inner product.  We define
\begin{equation}\label{eqn:aux}
 \overline V = V \times W_{1},
\end{equation}
where $W_{1}$ is an auxiliary (Hilbert) space with an inner product
$$
\overline a_{1}(\cdot,\cdot)=(\cdot,\cdot)_{A_1}.
$$
We also introduce an operator $\Pi_{1}:W_{1}\mapsto V$ in order to define the additive preconditioner, and define the latter as follows:
\begin{equation} \label{eqn:Ba}
\widehat{B} = \widetilde{S} + \Pi_{1}  A_{1}^{-1}\Pi_{1}^{T}.
\end{equation}

A distinctive feature of the auxiliary space method is the presence of $V$ in \eqref{eqn:aux} as a component of $\overline V$ and the presence of the symmetric positive definite operator $\widetilde{S}: V\mapsto V$.  $\widetilde{S}$ is assumed to be SPD in order to guarantee that $\widehat{B}$ is SPD and can be applied as a preconditioner for the PCG method.  It can be proved that the condition number of the preconditioned system can be bounded as follows~(see \cite{Xu.J1996}):
\begin{equation}\label{eqn:fasp}
\kappa(\widehat{B}A) \leq c_{0}^{2}(c_{s}^{2}+c_{1}^{2}), 
\end{equation}
if 
$$
\|\Pi_{1}w_{1}\|^2_A \leq c_{1}^2 \|w_{1}\|^2_{A_1}, \qquad w_{1}\in W_{1},
$$
and 
$$
(\widetilde{S}Av,v)_A \leq c_{s}^2(v,v)_A, \qquad  \forall v \in V.
$$
Moreover, for each $v\in V$, there exists $w_{1}\in W_{1}$ such that 
$$
v=v_{0}+\Pi_{1}w_{1}, \quad v_0 = v-\Pi_{1}w_{1}, 
\quad \text{and} \quad
(\widetilde{S}Av_{0},v_{0})_A + (w_{1},w_{1})_{A_1} \leq c_{0}^{2} (v,v)_A.
$$
Taking $W_1=V$ but with the inner product defined by $B$, $A_1^{-1}=B$, and $\Pi = I$, gives the additive version $\widehat{B}$ of the preconditioner $B_{\text{co}}$. Note that when $W_1=V$, $B$ is SPD and $\widetilde{S}$ is the symmetrization of the smoother $S$,  the preconditioner $\widehat{B}$ is SPD for both the norm-convergent smoother and the non-expansive smoother. 

\subsubsection{Derivation via block factorization}  
In this section, we present a derivation following the lines in~\cite[Chapter~5]{2008VassilevskiP-aa}. Let us introduce the operator
$\accentset{\bm{=}}{B}: V\times V\mapsto V\times V$ in block factored form:
\[
\accentset{\bm{=}}{B}=
\left(
\begin{matrix}
I & -S^TA\\
0 & I
\end{matrix}
\right)
\left(
\begin{matrix}
\widetilde{S}  & 0\\
0 & B
\end{matrix}
\right)
\left(
\begin{matrix}
I & 0\\
-AS & I
\end{matrix}
\right).
\]
We then set
\begin{equation}\label{eqn:barbarB}
\accentset{\bm{\approx}}{B}=
\begin{matrix}
& \\
(I,& I)\\ & 
\end{matrix}
\accentset{\bm{=}}{B}
\left(
\begin{matrix}
I \\I
\end{matrix}
\right), \qquad \accentset{\bm{\approx}}{B}:V \mapsto V.
\end{equation}
A straightforward calculation shows that $B_{\text{co}}$ and $\accentset{\bm{\approx}}{B}$ are the same. 
Regarding the positive definiteness of $\accentset{\bm{\approx}}{B}$, note that for the norm-convergent smoother, it is immediately evident that $\accentset{\bm{\approx}}{B}$ is SPD, because both $B$ and $\widetilde{S}$ are SPD.  For the non-expansive smoother, the positive-definiteness of $\accentset{\bm{\approx}}{B}$ does not immediately follow from the form of the preconditioner given in~\eqref{eqn:barbarB}. The arguments of Theorem~\ref{thm:B-is-SPD} (or similar) are needed to conclude that $\accentset{\bm{\approx}}{B}$ is SPD.

\subsection{Effectiveness of $B_{\text{co}}$\label{sect:compar}}
In this section, we show that the combined preconditioner, under suitable scaling assumptions performs no worse than its components.  As the numerical tests show (see Section~\ref{sec:numer}), the combination of ILU (for $B$) and AMG (for $S$) performs significantly better than its components. Let us set 
\begin{equation}\label{eq:m1m0def}
m_{1} = \lambda_{\max}(BA) \quad \mbox{and} \quad m_{0} =
\lambda_{\min}(BA). 
\end{equation}
Without loss of generality (with proper scaling), we can assume that the preconditioner $B$ is such that the following
inequalities hold:
\begin{equation}\label{eq:m1m0}
m_1 > 1 \geq m_0 > 0.
\end{equation}
We now prove a result that compares $\kappa(B_{\text{co}}A)$ with $\kappa(\widetilde{S}A)$ and $\kappa(BA)$ under the assumptions that $B$ and $S$ are such that both~\eqref{eq:m1m0} and \eqref{eq:rho} are satisfied.
\begin{theorem}\label{thm:compare}
If $S$ is a norm-convergent smoother and $B_{\text{co}}$ is defined as in~\eqref{eqn:barB}, then\begin{equation} \label{ine:cond}
\kappa(B_{\text{co}}A) \leq \frac{(1-m_1)(1-\rho)+m_1}{(1-m_0)(1-\rho)
  + m_0},
\end{equation}
and
\begin{equation}
\kappa(B_{\text{co}}A) < \kappa(BA).  \label{ine:better_B}
\end{equation}
Furthermore, if $S$ is such that~\eqref{eq:rho} holds with $\rho \geq 1 - \frac{m_0}{m_1 -1}$, then
\begin{equation}
\kappa(B_{\text{co}}A) \leq \kappa(\widetilde{S}A).  \label{ine:better_S}
\end{equation}
\end{theorem}

\begin{proof}
  From the assumption stated in~\eqref{eq:m1m0}, we immediately conclude that $B$ is SPD and $\kappa(BA) = m_1 / m_0$.  By the
  definition of $\widetilde{S}$, we have
\begin{equation*}
0 \leq ((I-\widetilde{S}A)w,w)_{A} = ((I-SA)w,  (I-SA)w)_{A} = \| (I-SA)w \|^{2}_{A} \leq \rho \| w \|^{2}_{A},
\end{equation*}
where we have used the assumption of the convergence of $S$ in the last inequality. By choosing $v = Aw$, we can obtain
\begin{equation} \label{ine:cond-smoother}
(1 - \rho) (A^{-1}v,v) \leq (\widetilde{S}v,v) \leq (A^{-1}v,v). 
\end{equation}
On the other hand, as $m_{1} = \lambda_{\max}(BA)$ and $m_{0} = \lambda_{\min}(BA)$, we have
$$
m_0 (A^{-1}v,v) \leq (Bv,v)  \leq m_1(A^{-1}v,v).
$$
By the definition of $B_{\text{co}}$ \eqref{eqn:tildeB1}, we have
\begin{align*}
(B_{\text{co}}v,v) & = (\widetilde{S}v,v) + (B(I-AS)v,(I-AS)v)   \\
		  & \leq (\widetilde{S}v,v) + m_1(A^{-1}(I-AS)v,(I-AS)v) \\
		  & = (1-m_1) (\widetilde{S}v,v) + m_1(A^{-1}v,v).
\end{align*}
Similarly, we can derive that
\begin{align*}
(B_{\text{co}}v,v) & = (\widetilde{S}v,v) + (B(I-AS)v,(I-AS)v)   \\
		  & \geq (\widetilde{S}v,v) + m_0(A^{-1}(I-AS)v,(I-AS)v) \\
		  & = (1-m_0) (\widetilde{S}v,v) + m_0(A^{-1}v,v).
\end{align*}
As $m_1 > 1 \geq m_0 $, we have
\begin{equation*}
[(1-m_0)(1-\rho)+m_0](A^{-1}v,v) \leq (B_{\text{co}}v,v) \leq [(1-m_1)(1-\rho) + m_1](A^{-1}v,v);
\end{equation*}
then the condition number of $B_{\text{co}}A$ is bounded by 
\begin{equation*}
\kappa(B_{\text{co}}A) \leq \frac{(1-m_1)(1-\rho)+m_1}{(1-m_0)(1-\rho) + m_0}.
\end{equation*}
If $m_{1} >1$, then 
$$
(1-m_{1})(1-\rho) + m_{1} < m_{1},
$$ 
and if $1 \geq m_{0}>0$, then 
$$
(1-m_{0})(1-\rho)+m_{0} \geq m_{0}.
$$ 
Note that 
$$
\kappa(B_{\text{co}}A) \leq \frac{(1-m_1)(1-\rho)+m_1}{(1-m_0)(1-\rho) + m_0} < \frac{m_{1}}{m_{0}} =\kappa(BA).
$$
Hence, the inequality~\eqref{ine:better_B}~holds if $m_1 > 1 \geq m_0 > 0$. 

On the other hand, \eqref{ine:better_S}~follows from 
\begin{equation} \label{ine:eq_better_S}
\frac{(1-m_1)(1-\rho)+m_1}{(1-m_0)(1-\rho) + m_0} \leq \frac{1}{1-\rho} = \kappa(\widetilde{S}A),
\end{equation}
where the last equality comes from \eqref{ine:cond-smoother}. Note that
\begin{equation*}
\frac{(1-m_1)(1-\rho)+m_1}{(1-m_0)(1-\rho) + m_0} = \frac{1 + \rho(m_{1}-1)}{(1-\rho) + \rho m_{0}}.
\end{equation*} 
Note that $\frac{a+c}{b+d} \leq \frac{a}{b}$ if $\frac{c}{d} \leq \frac{a}{b}$ and $a,b,c,d > 0$.  Therefore,\eqref{ine:eq_better_S}~holds if $\frac{m_{1}-1}{m_{0}} \leq \frac{1}{1-\rho}$, i.e., $\rho \geq 1 - \frac{m_{0}}{m_{1}-1}$.
\end{proof}

\begin{remark}
If $\frac{m_{1}-1}{m_{0}} < \frac{1}{1-\rho}$, then \eqref{ine:better_S} becomes $\kappa(B_{\text{co}}A) < \kappa(\widetilde{S}A)$.
\end{remark}

\begin{remark} \label{rem:compare}
If either $\widetilde{S}$ or $B$ works well as a preconditioner itself, there is no need to use the more complicated $B_{\text{co}}$.  However, we are interested in cases in which neither $\widetilde{S}$ nor $B$ works effectively alone.  In these cases, we use $B_{\text{co}}$ to combine these two, and we use Theorem~\ref{thm:compare}~to guarantees that $B_{\text{co}}$ will be a better preconditioner than either $B$ or $\widetilde{S}$ under reasonable conditions. The condition 
$$
\rho \geq 1 - m_0/(m_1 -1)
$$
means that both $\widetilde{S}$ and $B$ yield slow convergence as
$$
\rho \geq 1 - m_0/(m_1 -1) \approx  1 - \kappa(BA)^{-1} \quad \text{if} \  m_1 >> 1.
$$ 
Therefore, if $B$ is not a good preconditioner, i.e., $\kappa(BA)$ is large, then $\rho \approx 1$, which implies that $\widetilde{S}$ does not work well either.  The new preconditioner $\widetilde B$ is no worse than $B$ or $\widetilde S$ alone. In fact, the new preconditioner may be capable of performing much better than either $B$ or $\widetilde{S}$ based on our numerical experiments (see Section~\ref{sec:numer}). 
\end{remark}

\section{A simple example: ILU+AMG}
Algorithm~\ref{alg:barB} provides an approach to combining ILU and AMG methods. In the combined use of ILU and AMG for the linear system~\eqref{eqn:linear_system}, AMG serves as $S$ and ILU serves as $B$.  Next, we specify our choice of ILU and AMG for problems specific to reservoir simulation.  However, we would like to emphasize that the combined preconditioner works for a wide range of iterative methods and preconditioners as long as they satisfy the assumptions in Theorem~\ref{thm:B-is-SPD}.

\subsection{Incomplete LU factorization}
Incomplete LU factorizations compute a sparse lower triangular matrix $L$ and a sparse upper triangular matrix $U$ so that the residual matrix $R = A - LU$ satisfies certain conditions. A general algorithm of ILU can be obtained by Gaussian elimination and dropping some of the elements in the off-diagonal positions.  There are many variants of ILU preconditioners. They differ in terms of  the rules that govern the prescribed fill-in of the factors during the ILU factorization procedure.  For example, discarding the fill-in based on position gives the ILU($k$) method; discarding the fill-in based on the values of the corresponding entries in the $L$ or $U$ factors gives the \emph{threshold} ILU method.  There are also ILU methods for which the fill-in is managed based on a combination of positions and values or based on other dropping strategies.  For details about the different ILU methods, we refer to monographs \cite{Benzi.M2002,Chan.T;Van-der-Vorst.H1997} and Saad~\cite{Saad.Y2003}.  Here we use the notation and terminology from~Saad~\cite{Saad.Y2003}.

We consider ILU($k$) (Algorithm~\ref{alg:ILU_k}), which to our knowledge was originally introduced for reservoir simulations in \cite{Watts-III.J1981}.  The ILU is based on the \emph{level of fill} to determine the off-diagonal positions in the $L$ and $U$ factors where the entries fill-in will not be introduced (or as is often said, off-diagonal positions for which the fill-in entries are \emph{dropped}). Next, we define level of fill, and the detailed algorithm is given in
Algorithm~\ref{alg:ILU_k}.

\begin{definition}[Level of fill]
The initial level of fill of the elements of a sparse matrix $A$ defined as 
$$
\mathcal{L}_{ij} = 0 \ \text{if} \ a_{ij} \neq 0 \ \text{or} \  i=j, \ \text{otherwise} \ \mathcal{L}_{ij} = \infty.
$$
When an entry $a_{ij}$ is updated in the factorization procedure $a_{ij}:= a_{ij} - a_{ik}*a_{kj}$, its level of fill is also updated by
$$
\mathcal{L}_{ij} = \min\{\mathcal{L}_{ij}, \, \mathcal{L}_{ik}+\mathcal{L}_{kj}+1\}.
$$
\end{definition}

\begin{remark} \label{rem:level_of_fill}
Although the level of fill depends on the location of the element, the rationale is that the level of fill should indicate the magnitude of the element.  The higher the level of fill, the smaller the element.  For example, when the level of fill of $a_{ij}$ is $k$ this means that the size is $|a_{ij}|=O(\epsilon ^k)$ for some $\epsilon < 1$.
\end{remark}

\begin{algorithm}[h!]
\caption{ILU($k$) Method} \label{alg:ILU_k}
\begin{enumerate}
\item For all nonzero elements $a_{ij}$, define $\mathcal{L}_{ij} = 0$
\item For $i=1, \cdots n$, Do

          \hskip 0.15in For $m=1, \cdots, i-1$ and $\mathcal{L}_{im} \le k$, Do

            \hskip 0.3in Compute $a_{im} := a_{im}/a_{mm}$
            
            \hskip 0.3in Compute $a_{i*} := a_{i*} - a_{im}a_{m*}$, where $a_{i*}$ is the $i$-th row of $A$
            
            \hskip 0.3in Update the levels of fill of non-zero by $\mathcal{L}_{ij} = \min\{\mathcal{L}_{ij}, \, \mathcal{L}_{ik}+\mathcal{L}_{kj}+1\}$
            
          \hskip 0.15in End 

         \hskip 0.15in Replace any element in row $i$ with $\mathcal{L}_{ij} > k$ by zero

     \noindent    End
\end{enumerate}
\end{algorithm}

\subsection{Algebraic multigrid methods}
As an algebraic variant of MG methods, AMG methods are widely applicable and the focus of current intensive development.  AMG methods have mesh-independent convergence rates and optimal computational complexity for a wide range of problems.

Any AMG method consists of two phases: the SETUP phase and the SOLVE phase.  In the SETUP phase, the intergrid operator $P_l$ is constructed, and the coarse grid matrix is defined as 
$$
A_{l} = P_{l}^TA_{l+1}P_{l},  \quad  l=L-1, \ldots,0,
$$ 
and $A_L = A$. In the SOLVE phase, the smoother $S_l$ and coarse-grid correction are applied recursively as shown in the following general V-cycle MG Algorithm~\ref{alg:MG}.
\begin{algorithm}[htpb]
\caption{V-cycle for solving $A_l u_l = f_l$, with an initial guess $u_l^0$} \label{alg:MG}
\begin{enumerate}
\item {\bf Pre-smoothing:} $u_l^1 = u_l^0 + S_l (f_l - A_l u_l^0)$
\item {\bf Coarse-grid correction:}
	\begin{enumerate}
	\item $f_{l-1} = P_{l-1}^T(f_l - A_l u_l^1)$
	\item If $l = 1$, $e_0 = A_0^{-1} f_0$; else, apply Algorithm~\ref{alg:MG} for $A_{l-1}e_{l-1} = f_{l-1}$ with zero initial guess
	\item $u_l^2 = u_l^1 + P_{l-1}e_{l-1}$
	\end{enumerate} 
\item {\bf Post-smoothing:} $u_l^3 = u_l^2 + S_l^T (f_l - A_l u_l^2)$
\end{enumerate}
\end{algorithm}

In AMG methods, $P_l$ must be constructed based on algebraic principles, which presents certain challenges. The key to fast convergence is the complementary nature of operators $S_l$ and $P_l$. That is, errors not reduced by $S_l$ must be interpolated well by $P_l$.  We choose the classical AMG as the iterative method $S$ in our Algorithm~\ref{alg:barB}. This method constructs the intergrid operator $P_l$ in two steps.  First, the classical C-F splitting method is used, and then the operator $P_l$ is constructed by classical Ruge-St\"uben interpolation; see \cite{Ruge.J;Stuben.K1987} for details.

\subsection{Application of Algorithm~\ref{alg:barB}}
In this paper, we consider SPD coefficient matrices $A$ only.  In this case, ILU($k$) is replaced by an Incomplete Cholesky (IC) factorization. It is easy to see that IC($k$) is SPD, which satisfies the assumption on $B$ in Theorem~\ref{thm:B-is-SPD}.  For AMG methods, we have $A_{l} = P_{l}^TA_{l+1}P_{l}$, $l=L-1,\ldots,0$.  $A$ is SPD and $P_l$ is constructed in the classical AMG method; therefore, it is guaranteed that $A_{l}$ is also SPD for $l=0,1,\ldots,L$.  Hence, it is easy to see that using the standard Gauss-Seidel (GS) smoother on each level gives a convergent iterative method.  According to Algorithm~\ref{alg:MG}, $B_l$, which stands for V-cycle MG on level $l$ ($l=1,2,\ldots,L$), can be defined recursively as
\begin{equation*}
I-B_lA_l = (I-S_l^TA_l)(I-P_{l-1}B_{l-1}P_{l-1}^TA_{l})(I-S_lA_l),
\end{equation*} 
with $B_0 = A_0^{-1}$.  As
$$
\| I- S_l A_l \|_{A_l} < 1 \ \text{and} \
\| I-P_{l-1}B_{l-1}P_{l-1}^TA_{l} \|_{A_{l}} \leq 1,
$$ 
by induction assumption, we can show that 
$$
\| I-B_lA_l \|_{A_l} < 1.
$$ 
Thus the classical AMG is convergent for SPD problems by mathematical induction. Hence, ILU($k$) and the classical AMG can be combined using Algorithm~\ref{alg:barB} to yield the following corollary:
\begin{corollary}[Symmetric Positive Definiteness of ILU-AMG]
If we choose $S$ from the classical AMG methods and~$B$ from ILU($k$) in Algorithm~\ref{alg:barB}, then the operator $B_{\text{co}}$ defined in~\eqref{eqn:barB}~is SPD.
\end{corollary}

\section{A model problem in reservoir simulation}

Petroleum reservoir simulation provides information about processes that take place within oil reservoirs.  It is, therefore, of great assistance in efforts to achieve optimal recovery.  Modern reservoir simulation faces increasingly complex physical models and uses highly unstructured grids. This results in Jacobian systems that are more difficult at each step.  In this section, we describe a model problem in PRS and a discretization method that has great promise for efficient resource recovery, which is used for numerical tests in the next section.

Let $\Omega$ be a bounded domain in $\mathbb{R}^3$ in our consideration of the following second-order elliptic problem:
\begin{equation}\label{eqn:model}
-\nabla \cdot  (a \nabla p) + cp = f   \quad \text{in } \Omega.
\end{equation}
In reservoir simulation, the model problem~\eqref{eqn:model}~usually occurs after temporal semi-discretization of the mathematical models that describe the multiphase flow in porous media. Such scalar problems also occur in more sophisticated preconditioning techniques for coupled systems of PDEs (see~\cite{Xu.J2010} for further details).  The unknown function $p$ in~\eqref{eqn:model} is the pressure; $a \in [L^{\infty}(\Omega)]^{3 \times 3}$ is the diffusion tensor and it usually depends on the permeability and viscosity, etc.; $c \in L^{\infty}(\Omega)$ is nonnegative and proportional to the inverse of the time step size in an implicit temporal scheme (such as the
Backward Euler method). We assume that $\partial \Omega$ has two
non-overlapping parts, $\Gamma_D$ and $\Gamma_N$, such that
$\overline{\Gamma_D \cup \Gamma_N} = \partial \Omega$. The model
problem~\eqref{eqn:model}~is completed by the boundary conditions $p =
g_D$ on $\Gamma_D$ and $(a\nabla p) \cdot \pmb{n} = g_N$ on
$\Gamma_N$, where $\pmb{n}$ is the outward unit normal vector of
$\partial \Omega$, and where $g_D$ and $g_N$ are given.

The difficulties that arises from~\eqref{eqn:model} in regard to solving the linear system~\eqref{eqn:linear_system} are mainly due to the complex geometry and complicated physical properties of the reservoir, which often result in a heterogeneous diffusion tensor with large jumps, and in distorted, degenerated, and/or non-matching meshes with faults and pinch-outs.  Figure~\ref{fig:mesh} shows an example of the computational domain and the permeability of this reservoir in one horizontal layer.
\begin{figure}[htbp]
\begin{center}
\includegraphics[width=0.45\textwidth]{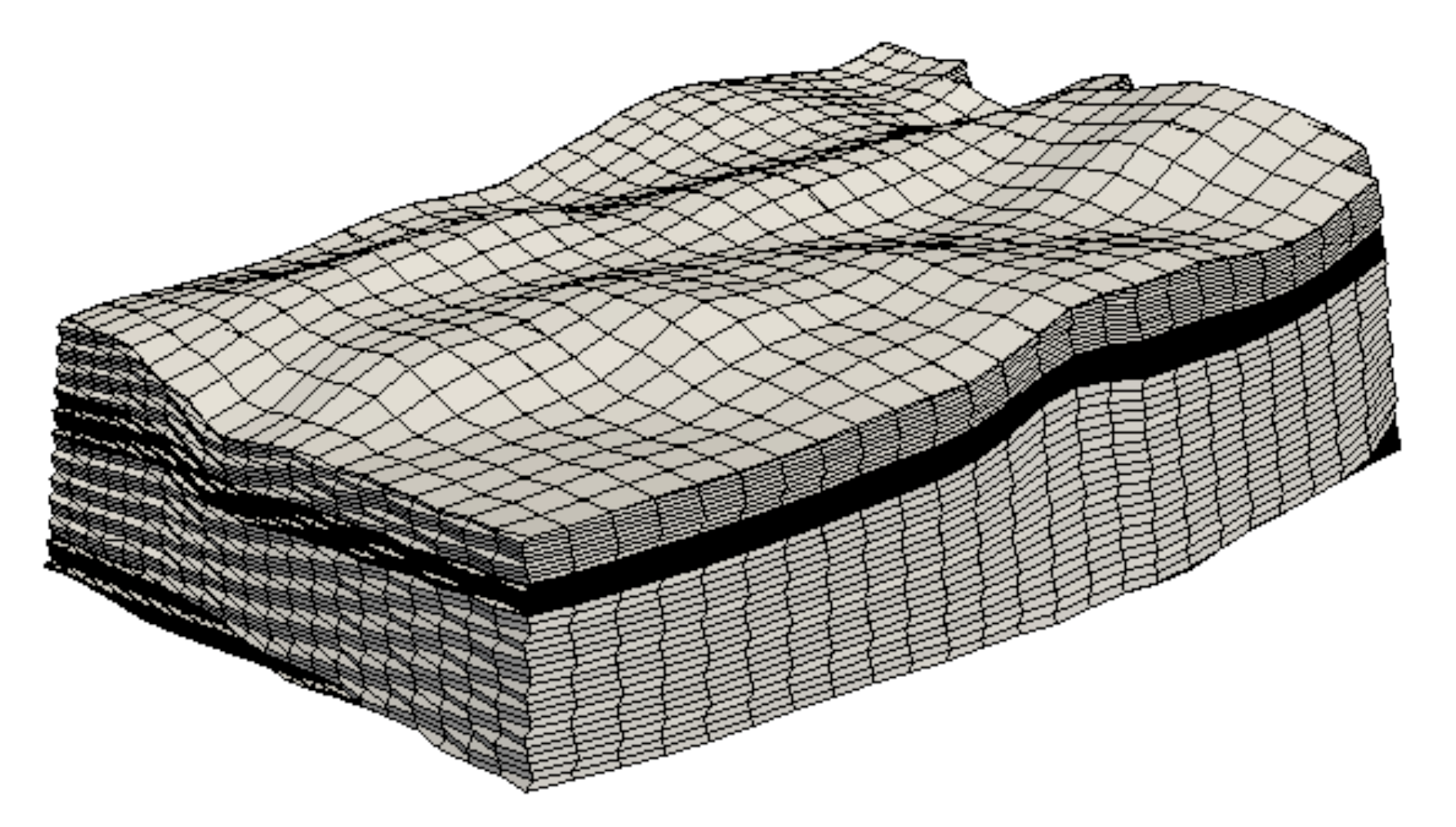} \hskip 0.46in \includegraphics[width=0.45\textwidth]{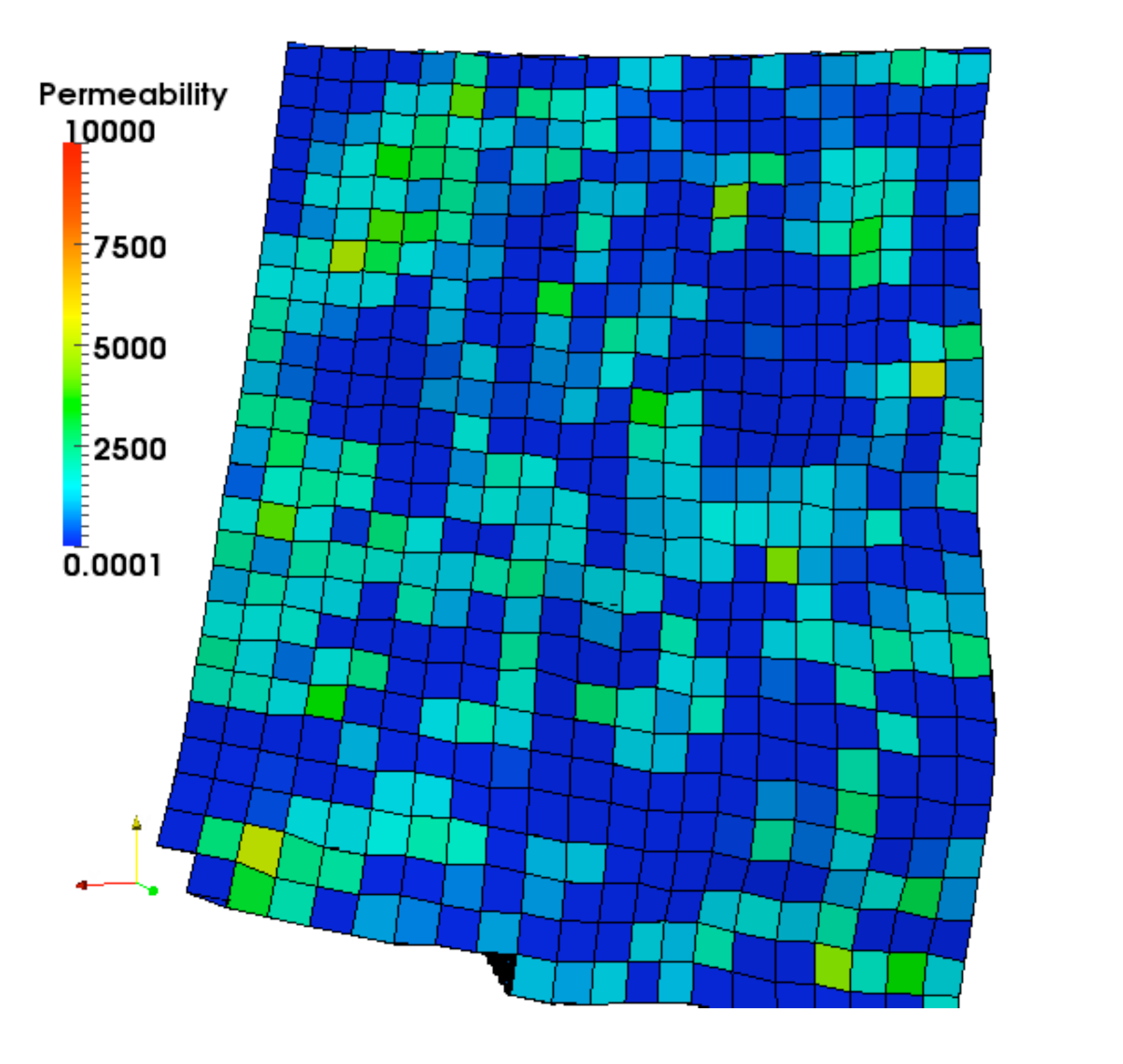} 
\caption{Left: a sample computational domain and mesh; Right: permeabilities of the porous media in one horizontal layer (the magnitude of the variation of permeability is about $10^8$).}
\label{fig:mesh}
\end{center}
\end{figure}

We use the mixed-hybrid finite element method to discretize~\eqref{eqn:model}.  Due to its local conservation property and intrinsic and accurate approximation of the flux $\pmb{u} := -a \nabla p$,  the mixed-hybrid method is preferred in reservoir simulations.  Let $\Omega_h =\{T\}$ be a triangulation of $\Omega$, and let $\Gamma_h$ be the set of boundaries of $T$ in $\Omega_h$ with the decomposition 
$$
\Gamma_h^{\partial} = \{  \Gamma \in \Gamma_h:  \Gamma \in \partial \Omega  \}, \quad \Gamma_h^0 = \Gamma_h \setminus \Gamma_h^{\partial}.
$$ 
We define the finite dimensional spaces as
\begin{align*}
V_h &  := \lbrace \pmb{v} \in [L^2(\Omega)]^d : \pmb{v}\vert_{T} \in V_h(T), \; \forall\, T \in \Omega_h \rbrace
\\ 
Q_h & := \lbrace  q \in L^2(\Omega) : q\vert_{T} \in P_0(T), \; \forall\, T \in \Omega_h \rbrace
\\ 
\Lambda_h & := \lbrace  \mu \in L^2(\Gamma_h): \mu \vert_{\Gamma} \in P_0(\Gamma), \; \forall \, \Gamma \in \Gamma_h^0; \; \mu \vert_{\Gamma} = 0, \; \forall\,\Gamma \in \Gamma_h^{\partial}  \rbrace,
\end{align*}
where $V_h(T)$ is the Kuznetsov-Repin element on the polyhedral elements \cite{Kuznetsov.Y;Repin.S2003}. This element is based on the Raviart-Thomas element on the local conforming tetrahedral partitioning of $T$. $P_k(T)$ denotes the set of polynomials of a degree not greater than $k$, $k\ge 0$. $V_h$, $Q_h$, and $\Lambda_h$ are used to approximate the flux $\pmb{u}$; the pressure $p$ associated with elements in $\Omega_h$; and the Lagrange multipliers $\lambda_h$ are associated with faces in $\Omega_h$. 

The mixed-hybrid finite element formulation for \eqref{eqn:model} can be written as follows: Find $(\pmb{u}_h, p_h, \lambda_h)\in
V_h\times Q_h\times \Lambda_h$, $\pmb{u}_h \cdot \pmb{n} = -g_{N,h}$
on $\Gamma_h$, where $g_{N,h}$ is an appropriate approximation of $g_N$, such that
\begin{align*}
(a^{-1}\pmb{u}_h, \pmb{v}) - \sum_{T}\lbrace  (p_h, \text{div} \ \pmb{v})_{T} - (\lambda_h, \pmb{v} \cdot \pmb{n}_{T})_{\partial T} \rbrace &= - \int_{\Gamma_D} g_D(\pmb{v} \cdot \pmb{n}) \mathrm{d}s,  & &\forall \pmb{v} \in V_h
\\
-\sum_{T}(\text{div} \ \pmb{u}_h, q)_{T} - (cp_h, q) &= - \int_{\Omega} fq \mathrm{d}x, & &\forall q \in Q_h
\\
\sum_{T} (\mu, \pmb{u}_h \cdot \pmb{n}_{T})_{\partial T} &= 0, & &\forall \mu \in \Lambda_h
\end{align*}
where $\pmb{n}_T$ denotes the outward unit normal to $\partial T$. 

\section{Numerical Experiments}\label{sec:numer}
It is easy to see that the resulting discrete linear system of the mixed-hybrid finite element method in the previous section has the following matrix form:
\begin{equation*}
\begin{pmatrix}
\mathsf D & \mathsf  B^T & \mathsf  C^T \\
\mathsf B & - \mathsf M   &  0   \\
\mathsf C & 0     &  0
\end{pmatrix}
\begin{pmatrix}
\widetilde{u} \\
\widetilde{p} \\
\widetilde{\lambda}
\end{pmatrix} =
\begin{pmatrix}
\widetilde{f_1} \\
\widetilde{f_2} \\
0
\end{pmatrix}.
\end{equation*}
Notice that the matrix $\mathsf D$ is block diagonal with each block corresponding to the flux unknowns in one element. Hence, we can easily invert $\mathsf D$ and obtain~\eqref{eqn:linear_system} with 
\begin{equation} \label{eqn:alge_sys}
A = \begin{pmatrix}
\mathsf{BD^{-1}B}^T + \mathsf{M} & \mathsf{BD^{-1}C}^T \\
\mathsf{CD^{-1}B}^T        &  \mathsf{CD^{-1}C}^T
\end{pmatrix}
, \quad
u = 
\begin{pmatrix}
\widetilde{p}\\
\widetilde{\lambda}
\end{pmatrix}, \quad \text{and} \
f = 
\begin{pmatrix}
\mathsf{BD^{-1}}\widetilde{f_1} - \widetilde{f_2} \\
\mathsf{CD^{-1}} \widetilde{f_1}
\end{pmatrix},
\end{equation}
where it is well-known that $A$ is SPD. 

In the numerical experiments, performed on a Dell Precision desktop computer, we solve the above linear system by PCG with different preconditioners.  We pick eight problems from the real petroleum reservoir data. Table~\ref{tab:size} gives the degrees of freedom (DOFs) and the number of non-zeros (NNZ) for the test problems.  Here, the difference between Model 1 (3) and 2 (4) is that the permeabilities in Model 1 (3) are homogeneous whereas the permeabilities in Model 2 (4) have large jumps. Figure~\ref{fig:mesh} shows the computational domain of Models 3 and 4, and the highly heterogenous permeabilities used in Model 4.  Models 5--8 are test problems that are relatively large in size.  The computational domain of Models 5 and 6 are the same, but the physical properties, such as porosity and permeability, are different.  Models 7 and 8 share the same computational domain, but the permeabilities used in Model 7 are from real reservoir data and the permeabilities in Model 8 are artificially adjusted in order to make the resulting linear system very difficult to solve.

\begin{table}[htbp]
\centering
\caption{Degree of freedom (DOF) and number of non-zeros (NNZ) of the model problems.}
\label{tab:size}
\begin{tabular}{|c||c|c|c|c|}
    \hline
         Test Problem Set 1 & Model 1 & Model 2 & Model 3 & Model 4 \\
         \hline
    DOF & 156,036 & 156,036  & 287,553 &   287,553 \\
    \hline
    NNZ & 1,620,356 & 1,620,356  & 3,055,173 & 3,055,173 \\
    \hline
    \hline
    Test Problem Set 2 & Model 5 & Model 6 & Model 7 & Model 8 \\
    \hline
    DOF & 1,291,672 & 1,291,672  & 4,047,283 &   4,047,283 \\
    \hline
    NNZ & 13,930,202 & 13,930,202  & 45,067,789 & 45,067,789 \\
    \hline
\end{tabular}
\end{table}
\begin{table}[ht]
\def\temptablewidth{1\textwidth}
\centering
\caption{Comparison of preconditioners for model problems (stopping criteria: relative residual less than $10^{-10}$). In the table, `---' means the method does not converge within 10000 iterations, \#Iter is the number of iterations, and the unit for CPU time is second.}
\begin{tabular*}{\temptablewidth}{@{\extracolsep{-1.75mm}}|l||r|r|r|r|r|r|r|r|}
    \cline{1-9}
    Test Problem Set 1 &			\multicolumn{2}{c|}{Model 1} &  \multicolumn{2}{c|}{Model 2} &  \multicolumn{2}{c|}{Model 3} &  \multicolumn{2}{c|}{Model 4} \\ \cline{1-9}
   Preconditioner    		               & \#Iter  & CPU     & \#Iter     & CPU   & \#Iter & CPU      & \#Iter & CPU\\
     \cline{1-9}
     
         ILU(0)   				      & 234  & 2.85    & 3034     & 32.91  & 291       & 6.56           & 4939    & 102.13 \\ \cline{1-9}   
       
     AMG(GS)     		      & 135       & 13.16      & 136       & 13.48      & 154 	    & 31.70    & 138      & 25.78  \\ \cline{1-9} %
     
     AMG(ILU($0$)) 	      &  5364   & 631.72   & 1058    & 115.71 & ---   & ---    & 4849      & 1024.21 \\ \cline{1-9}   %
    
     $\widehat{B}$  (Additive)  	     &	   44     &   6.18    & 43          & 5.53   & 44       & 11.40    & 49      & 11.33 \\  \cline{1-9} %
     
     $B_{\text{co}}$ (Multiplicative) &   23     &   4.04    & 25          & 3.92    & 22       & 7.36    & 28      & 7.78 \\ \hline \hline 
     
         Test Problem Set 2 &			\multicolumn{2}{c|}{Model 5} &  \multicolumn{2}{c|}{Model 6} &  \multicolumn{2}{c|}{Model 7} &  \multicolumn{2}{c|}{Model 8} \\ \cline{1-9}
   Preconditioner    		               & \#Iter  & CPU     & \#Iter     & CPU   & \#Iter & CPU      & \#Iter & CPU\\
     \cline{1-9}
     ILU(0)   				      &  ---  &  ---    & ---     & ---  & 838       & 271.46           & 2039    & 646.77 \\ \cline{1-9}   
       
     AMG(GS)     		      &   30     & 36.70      & 30       & 37.07      & 19  & 96.03    &   22    &  112.79  \\ \cline{1-9} %
     
     AMG(ILU($0$)) 	      &  15    & 32.16   &  15  & 31.58  & 11 	    & 96.72    & 11      & 103.52\\ \cline{1-9}   %
    
     $\widehat{B}$  (Additive)  &   32     &   39.72    &     32      & 39.87   & 26       & 100.78    & 25      & 104.29 \\  \cline{1-9} %
     
     $B_{\text{co}}$ (Multiplicative) &   19     &   25.33    & 19          & 25.54    & 14       & 76.92    & 14      & 77.88 \\  \cline{1-9}
     
\end{tabular*}
\label{tab:compare}
\end{table}

%
%
%
%
%
%

Table~\ref{tab:compare}~compares different preconditioners for the model problems. The CPU time listed in Table~\ref{tab:compare}~is the total CPU time which includes both the setup and solver times. We consider only the ILU($0$) preconditioner due to memory-usage considerations.  As we use highly unstructured meshes, the amount of fill-in for obtaining ILU($k$) factorizations is not predictable for $k>0$.  Here, the AMG(GS) refers to the V-cycle MG method with standard GS smoother on each level. The AMG(ILU($0$)) refers to the V-cycle MG method with an ILU($0$) smoother on the finest level and a GS smoother on the other levels.  In both cases,  two cycles are applied in order to define the preconditioner.  $B_{\text{co}}$ is defined in~\eqref{eqn:barB} with one V-cycle MG, a GS smoother on each level serves as the smoother $S$, and ILU($0$) serves as the preconditioner $B$.  $\widehat{B}$ is its additive version defined in~\eqref{eqn:Ba} with the special choices of $J=1$, $W_1 = V$, $\Pi_1 = I$ and $A_1^{-1}$ is substituted by the preconditioner $B$. 

The ILU($0$) preconditioner works well for Models 1 and 3 because these two problems are relatively small in size and their permeabilities are homogeneous.  We can see that the ILU($0$) preconditioner does not work well for problems that are highly heterogeneous (Models 2 and 4).  It is also not efficient for large-size problems (Models 5--8).  Its performance is not predictable, and sometimes it may even break down (Models 5 and 6).  AMG with a standard GS smoother works better than ILU($0$) does, but still requires more than $100$ iterations for Models 1--4.
If we use ILU($0$) as the smoother at the finest level, the resulting AMG preconditioner may not be SPD because ILU($0$) might not converge. From the numerical results, we can see that  for Models 1--4, using AMG with the ILU($0$) smoother does not work as fast as using AMG with the GS smoother does.  The former method may break down due to the lack of an SPD property (Model 3).  However, if we combine AMG and ILU($0$) as in Algorithm~\ref{alg:barB}, the new defined preconditioner $B_{\text{co}}$ with AMG and ILU($0$) gives the best performance in terms of CPU time; and, it performs efficiently and robustly with respect to the problem size and heterogeneity. The additive algorithm $\widehat{B}$ also works efficiently and robustly. The numerical tests confirm our theoretical results and show the efficiency of our new approach for practical reservoir simulation problems.

\

\section{Concluding Remarks}
In this paper, we discussed practical and efficient solvers for large sparse linear systems and their applications in petroleum reservoir simulation. We propose a simple and efficient framework for obtaining a new preconditioner by combining a convergent iterative method (such as an AMG method with a standard smoother) with an existing preconditioner (such as an ILU method).  We used this approach to solve large linear systems arising from petroleum reservoir simulation (with highly heterogeneous media and 3D unstructured grids).  And, our numerical results show that the new preconditioners significantly improve the efficiency of the linear solvers and are robust with respect to the size and heterogeneity of practical problems.  This method has been implemented and employed in real reservoir simulation by ExxonMobil Upstream Company and RIPED, PetroChina.

\section*{Acknowledgement}
The authors wish to thank Dr. Robert Falgout for many helpful discussions.

\bibliographystyle{abbrv}
\bibliography{ILU_AMG_EM}

\end{document}